\PassOptionsToPackage{dvipsnames}{xcolor}
\documentclass{amsart}
\usepackage{amsmath,amssymb,amsthm}
\usepackage[indent=0pt,skip=6pt plus 1pt]{parskip}
\usepackage{etoolbox}
\makeatletter
\@ifclassloaded{amsart}{%
  \patchcmd{\@maketitle}{\dimen@34\p@}{\dimen@22\p@}{}%
    {\PackageError{independence}{AMS title spacing patch failed}{}}%
}{}
\makeatother
\usepackage{xcolor}
\usepackage{hyperref}
\hypersetup{colorlinks=true,linkcolor=MidnightBlue,citecolor=MidnightBlue,
            urlcolor=MidnightBlue}
\usepackage{url}

\AddToHook{cmd/@parboxrestore/after}{%
  \setlength{\parindent}{0pt}%
  \setlength{\parskip}{6pt plus 1pt}%
}

\newtheorem{theorem}{Theorem}[section]
\newtheorem{corollary}[theorem]{Corollary}

\theoremstyle{remark}

\title[Annihilation-number Caro--Wei bound]{An annihilation-number Caro--Wei bound:\\
       a TxGraffiti conjecture and an independence-number bracket}
\author[Chakshu Gupta]{Chakshu Gupta}
\address{College of Computing, Georgia Institute of Technology}
\email{cgupta65@gatech.edu}
\date{}

\begin{document}
\raggedbottom

\begin{abstract}
Computable independence bounds give intervals that can be compared by
containment or by their worst-case endpoint ratios. For positive maximum degree
$\Delta$, the Caro--Wei sum $W$ bounds the independence number from below,
while $(\Delta+1)W/2$ bounds it from above. A second interval has the
Havel--Hakimi residue as its lower endpoint and the annihilation number
as its upper endpoint.
This note proves that the annihilation number does not exceed the
Caro--Wei upper endpoint. This upper bound is sharp, with equality for
regular graphs of positive degree and even order. Since the residue is
at least $W$, this upper bound places the residue--annihilation interval
inside the Caro--Wei interval. Both endpoints can improve strictly on one graph,
although the sharp worst-case endpoint ratio remains unchanged, even for
connected graphs. The annihilation-number bound also proves a conjectured
lower estimate based on the annihilation number and residue for every
connected graph with at least three vertices. Equality in this estimate
holds for exactly eight graph types, while the single edge violates it.
For maximum degree at least three, however, the estimate never exceeds
the residue, even without connectedness.
\end{abstract}

\maketitle

\section{Introduction}\label{sec:intro}

Degree-sequence bounds enclose a graph's independence number in an interval.
Comparing such intervals by containment need not give the same conclusion
as comparing their worst-case endpoint ratios. This note examines that
distinction for intervals based on the Havel--Hakimi
residue, the annihilation number and the Caro--Wei sum.

The Caro--Wei interval uses the degrees $d(v)$ of a finite simple graph
$G$ to bound its independence number $\alpha(G)$. Its lower endpoint is
the Caro--Wei sum $W(G)$~\cite[Equation~(1)]{Boppana2020}. For positive
maximum degree $\Delta(G)$, the known upper estimate~\cite[Theorem~1]{Boppana2020}
supplies the other endpoint.
\begin{align*}
  W(G)&=\sum_{v\in V(G)}\frac{1}{d(v)+1},\\
  W(G)&\le\alpha(G)\le\frac{\Delta(G)+1}{2}W(G).
\end{align*}

The other interval replaces these endpoints by the residue $R(G)$ and
annihilation number $a(G)$.
The residue counts the zeros left by the degree-sequence reduction
originating in graph-realization criteria~\cite{Hakimi1962,Havel1955}.
Each step removes a largest positive entry, subtracts one from that many
largest remaining entries, and sorts again; zeros are retained throughout
\cite[Section~0]{Favaron1991}. The annihilation number instead records the
largest integer $k$ for which the $k$ smallest degrees of $G$ sum to at most the number
of edges~\cite[Definition~2.3.1]{Pepper2004}.
The classical residue bound~\cite[Theorem~3.3]{Favaron1991} and
annihilation bound~\cite[Theorem~2.3.2]{Pepper2004} place independence
in the residue--annihilation interval.
\[
  R(G)\le\alpha(G)\le a(G).
\]

The residue is at least the Caro--Wei sum~\cite[Theorem~1.1]{Favaron1991},
so interval containment reduces to comparing the upper endpoints.
The Caro--Wei upper estimate alone does not control the annihilation
number, which can exceed the independence number. For positive maximum
degree, combining classical degree-sum and reciprocal-degree estimates
supplies the needed upper bound. The resulting containment improves
both endpoints strictly on the star $K_{1,3}$. The endpoint-ratio bound
nevertheless remains sharp, attained by connected graphs of every
positive maximum degree.
\begin{align*}
  a(G)&\le\frac{\Delta(G)+1}{2}W(G),\\
  \frac{a(G)}{R(G)}&\le\frac{\Delta(G)+1}{2}.
\end{align*}

Subsequent work refines the annihilation-number bound proved here into
an exact nonnegative-slack identity~\cite[Theorem~7.3]{KadrawiLevit2026},
with degree-sequence equality criteria~\cite[Corollary~7.4]{KadrawiLevit2026}.
Its independence bracket has the annihilation number as its upper endpoint
and the maximum of the residue, the Caro--Wei sum and a matching-based
estimate as its lower endpoint~\cite[Corollary~7.5]{KadrawiLevit2026}.
The comparison here instead concerns containment in the Caro--Wei interval
and the sharp endpoint ratio attained by connected graphs.

The annihilation-number bound also applies to a conjecture of TxGraffiti, an automated
conjecturing program. The conjecture asserts
$\alpha(G)\ge(a(G)+R(G))/\Delta(G)$ for every nontrivial connected
graph~\cite[Conjecture~1]{Davila2025reverie}.
Combining the annihilation-number bound with the classical residue
inequality proves the case above maximum degree two; paths and cycles
supply the degree-two case. The inequality holds for every connected
graph with at least three vertices, with exactly eight equality types
and $K_2$ as the only nontrivial connected exception. As a lower estimate,
however, it never exceeds $R(G)$ at maximum degree at least three,
even without connectedness.

The gap between this estimate and the residue has an exact identity in
subsequent work~\cite[Corollary~7.7]{KadrawiLevit2026}.
That work also proves the TxGraffiti inequality independently without the
Caro--Wei sum~\cite[Theorem~7.10]{KadrawiLevit2026}.
Its equality criterion is stated in terms of graph
invariants~\cite[Corollary~7.8]{KadrawiLevit2026}, whereas the classification
here lists all eight isomorphism types attaining equality.

\section{The annihilation bound and interval comparison}\label{sec:results}

\begin{theorem}[Annihilation-number Caro--Wei bound]\label{thm:vehicle}
Every finite simple graph $G$ with maximum degree $\Delta(G)\ge1$
satisfies $a(G)\le(\Delta(G)+1)W(G)/2$.
Equality is attained whenever $G$ is regular and has even order.
\end{theorem}

\begin{proof}
The graph $G$ has $n$ vertices and $m$ edges, with degrees
$d_1,\ldots,d_n$ in nondecreasing order. Its annihilation number is
$a=a(G)$ and its maximum degree is $\Delta=\Delta(G)$.
By definition, the first $a$ degrees sum to at most $m$.
The remaining degrees then sum to at least $m$, since all degrees sum
to $2m$. Bounding each remaining degree by $\Delta$ gives the
classical degree-sum estimate~\cite[Theorem~2.3.3]{Pepper2004}.
\begin{align*}
  m&\le\sum_{i=a+1}^{n}d_i\le\Delta(n-a),\\
  a&\le n-\frac m\Delta.
\end{align*}

\medskip\noindent
The upper bound $n-m/\Delta$ can now be compared with $(\Delta+1)W(G)/2$.
The reciprocal-degree estimate in~\cite[proof of Corollary~1.3]{Favaron1991}
rewrites their difference using $\sum_i d_i=2m$. Putting each summand
over a common denominator and factoring its numerator gives
\begin{align*}
  \frac{\Delta+1}{2}W(G)-\left(n-\frac m\Delta\right)
    &=\frac12\sum_{i=1}^{n}
      \left[\frac{\Delta+1}{d_i+1}
        -\left(2-\frac{d_i}{\Delta}\right)\right]\\
    &=\frac1{2\Delta}\sum_{i=1}^{n}
      \frac{(\Delta-d_i)(\Delta-d_i-1)}{d_i+1}.
\end{align*}

\medskip\noindent
The integer degrees make each factored numerator nonnegative. If $d_i=\Delta$,
it is zero; otherwise the integer degree satisfies $d_i\le\Delta-1$,
so both factors are nonnegative. The positive denominators preserve this
sign in the sum. Combining the resulting reciprocal-degree
comparison with the degree-sum bound proves the required inequality.
\[
  a(G)\le n-\frac m\Delta\le\frac{\Delta+1}{2}W(G).
\]

\medskip\noindent
This bound is attained by every $\Delta$-regular graph of even order $n$.
For such a graph, the annihilation condition $k\Delta\le n\Delta/2$
makes half the order the largest allowable $k$.
Regularity also fixes every term of the Caro--Wei sum, so these two values
meet at the upper bound.
\begin{align*}
  W(G)&=\frac{n}{\Delta+1},\\
  a(G)&=\frac n2=\frac{\Delta+1}{2}W(G).
\end{align*}
\end{proof}

The positive-degree hypothesis is necessary. An edgeless graph on
$n\ge1$ vertices has $a(G)=n$ and $W(G)=n$. Its proposed upper bound
is $n/2$, strictly smaller than its annihilation number.

\begin{samepage}
\begin{corollary}[Bracketing the independence number]\label{cor:approx}
For every finite simple graph $G$ with maximum degree $\Delta\ge1$,
the residue--annihilation interval is contained in the Caro--Wei interval.
Its endpoint ratio satisfies $a(G)/R(G)\le(\Delta+1)/2$, with equality
attained by connected graphs of every positive maximum degree.
The independence bounds are
\[
  W(G)\le R(G)\le\alpha(G)\le a(G)
    \le\frac{\Delta+1}{2}W(G).
\]
\end{corollary}
\end{samepage}

\begin{proof}
The classical residue and annihilation inequalities give the first three
links. The annihilation-number bound of Theorem~\ref{thm:vehicle}
supplies the last and completes the interval containment. Since
$W(G)>0$ and $R(G)\ge W(G)$, division by the residue then gives the
endpoint-ratio bound.
\[
  \frac{a(G)}{R(G)}
    \le\frac{\Delta+1}{2}\frac{W(G)}{R(G)}
    \le\frac{\Delta+1}{2}.
\]

\medskip\noindent
The single edge attains this ratio for $\Delta=1$, since its annihilation
number and residue both equal one. For $\Delta\ge2$, deleting a perfect matching from
$K_{\Delta+1,\Delta+1}$ gives a $\Delta$-regular graph $G_\Delta$ on
$2\Delta+2$ vertices. Two vertices in the same part have $\Delta-1$
common neighbors, namely every vertex in the opposite part except
their two matched partners. Thus vertices in each part are joined by
two-edge paths, and an edge between the parts makes $G_\Delta$
connected.

\medskip\noindent
The connected graph $G_\Delta$ is regular of even order, so
Theorem~\ref{thm:vehicle} gives $a(G_\Delta)=\Delta+1$ and $W(G_\Delta)=2$.
Its residue can be determined from the disjoint union $H$ of two copies
of $K_{\Delta+1}$, which has the same degree sequence.
The Caro--Wei sum and residue therefore agree for the two graphs.
The two cliques give $W(H)=2$ and $\alpha(H)=2$, so the classical residue
bounds force $R(H)=2$. The common degree sequence then gives
$R(G_\Delta)=2$, and the endpoint ratio is
$(\Delta+1)/2$ as claimed.
\end{proof}

This sharp ratio is compatible with strict improvement of both endpoints
on an individual graph. The star $S=K_{1,3}$ has ordered degrees
$1,1,1,3$, giving $a(S)=3$. Its three leaves form an independent set;
the annihilation upper bound therefore gives $\alpha(S)=3$.
One Havel--Hakimi step leaves three zeros, giving $R(S)=3$.
By contrast, the leaf and central degrees give $W(S)=7/4$ and the
Caro--Wei upper endpoint $7/2$. The residue--annihilation interval is
therefore strictly narrower at both endpoints.
\begin{align*}
  [R(S),a(S)]&=[3,3]\\
    &\subsetneq\left[\frac74,\frac72\right]
      =\left[W(S),\frac{\Delta(S)+1}{2}W(S)\right].
\end{align*}

\section{The TxGraffiti application}\label{sec:txgraffiti}

The same interval endpoints $a(G)$ and $R(G)$ also enter TxGraffiti
Conjecture~1. It proposes $\alpha(G)\ge(a(G)+R(G))/\Delta(G)$ for
every nontrivial connected graph~\cite{Davila2025reverie}.
The single edge $K_2$ violates this proposed inequality.
Its maximum degree, independence number, annihilation number and
residue all equal one, so the conjectured lower bound on independence
is two.

\begin{theorem}[TxGraffiti inequality]\label{thm:main}
Every finite simple connected graph $G$ with $\Delta(G)\ge2$
satisfies $\Delta(G)\alpha(G)\ge a(G)+R(G)$.
Equality is attained at maximum degrees two and three.
\end{theorem}

\begin{proof}
The annihilation-number bound controls $a(G)$ above maximum degree two;
at degree two, the path and cycle degree sequences instead give
$a(G)=\alpha(G)$. In both cases, the residue bound completes the
comparison.

\medskip\noindent
\emph{Maximum degree at least three.}
The maximum degree $\Delta=\Delta(G)$ is at least three, so
$(\Delta+1)/2$ is at most $\Delta-1$.
Combining this coefficient comparison with Theorem~\ref{thm:vehicle} and the classical
bounds $W(G)\le\alpha(G)$ and $R(G)\le\alpha(G)$ gives
\begin{align*}
  a(G)+R(G)&\le\frac{\Delta+1}{2}W(G)+R(G)\\
    &\le(\Delta-1)\alpha(G)+\alpha(G)\\
    &=\Delta\alpha(G).
\end{align*}

\medskip\noindent
\emph{Maximum degree two.}
If $\Delta(G)=2$, connectedness makes $G$ a path $P_n$ or a cycle
$C_n$ for some $n\ge3$. The path has two degree-one vertices, all
remaining degrees two, and $n-1$ edges. Its first two degrees satisfy
the annihilation condition; for $k\ge2$, its first $k$ degrees sum
to $2k-2$. Maximizing $k$ under $2k-2\le n-1$ fixes its annihilation
number, attained by an independent set taking alternate vertices from
an endpoint. In the cycle case, all degrees are two and there are $n$ edges, so the
annihilation constraint is $2k\le n$. An independent set attaining
the resulting maximum takes alternate vertices among $n-1$ consecutive
vertices. These constructions and the classical annihilation upper
bound give the path and cycle values below. Adding the residue bound
then proves the required inequality.
\begin{align*}
  a(P_n)=\alpha(P_n)&=\left\lceil\frac n2\right\rceil,\\
  a(C_n)=\alpha(C_n)&=\left\lfloor\frac n2\right\rfloor,\\
  a(G)+R(G)&\le2\alpha(G)=\Delta(G)\alpha(G).
\end{align*}

\medskip\noindent
\emph{Attainment.}
On $P_3$, the independence number, annihilation number and residue all
equal two. On $C_3$, all three equal one. Both graphs therefore attain
equality at maximum degree two. At maximum degree three, the complete
graph $K_4$ has annihilation number two, and independence number and
residue both one, giving equality at maximum degree three.
\end{proof}

\begingroup
\relpenalty=10000
\binoppenalty=10000
The proof of Theorem~\ref{thm:main} does not use connectedness when the
maximum degree is at least three. At maximum degree two, however,
connectedness cannot be omitted. The disconnected graph
$G=C_3\cup K_2$~\cite[Proposition~7.11]{KadrawiLevit2026} has four edges and ordered
degrees $1,1,2,2,2$. The three smallest degrees sum to four, but the
four smallest sum to six, so \mbox{$a(G)=3$}. Each clique component
contributes one to both the independence number and the Caro--Wei sum.
Both quantities therefore equal two, forcing $R(G)=2$ by the
classical bounds. Since $\Delta(G)=2$, the inequality fails.
\[
  a(G)+R(G)=5>4=\Delta(G)\alpha(G).
\]
\endgroup

\begin{samepage}
\begin{theorem}[Equality and strictness]\label{cor:sharp}
For every finite simple connected graph $G$ with maximum degree
$\Delta\ge2$, the equality $\Delta\alpha(G)=a(G)+R(G)$ holds if and
only if $G$ is isomorphic to one of
$P_3$, $P_4$, $P_6$, $C_3$, $C_4$, $C_5$, $C_7$ and $K_4$.
No connected graph with $\Delta\ge4$ attains equality.
The equality $\Delta\alpha(G)=a(G)+R(G)$ also holds if and only if
both conditions below hold.
\begin{align*}
  a(G)&=(\Delta-1)\alpha(G),\\
  R(G)&=\alpha(G).
\end{align*}
\end{theorem}
\end{samepage}

\begin{proof}
The two displayed conditions require the annihilation and residue bounds
to be tight separately. Their simultaneous tightness forces a clique
above degree two; at degree two, the residue distinguishes the remaining
paths and cycles.

\medskip\noindent
\emph{Invariant criterion.}
The proof of Theorem~\ref{thm:main} gives
$a(G)\le(\Delta-1)\alpha(G)$ in both degree regimes. The classical
residue inequality gives $R(G)\le\alpha(G)$, so both gaps are
nonnegative. Their sum is the gap in the TxGraffiti inequality, which
therefore vanishes exactly when both gaps vanish.
\begin{align*}
  \Delta\alpha(G)-a(G)-R(G)
    &=\bigl((\Delta-1)\alpha(G)-a(G)\bigr)\\
    &\quad+\bigl(\alpha(G)-R(G)\bigr).
\end{align*}

\medskip\noindent
\emph{Higher-degree reduction.}
For an equality graph $G$ with maximum degree $\Delta\ge3$, the
invariant criterion gives $a(G)=(\Delta-1)\alpha(G)$. The annihilation-number
bound and Caro--Wei place this value in a chain with identical endpoints.
Every link must therefore be tight. Since $\alpha(G)>0$, tightness of the coefficient
comparison forces $\Delta=3$, and tightness of the Caro--Wei comparison
gives $W(G)=\alpha(G)$ in
\begin{align*}
  (\Delta-1)\alpha(G)=a(G)
    &\le\frac{\Delta+1}{2}W(G)\\
    &\le\frac{\Delta+1}{2}\alpha(G)
      \le(\Delta-1)\alpha(G).
\end{align*}

\medskip\noindent
\emph{Caro--Wei equality.}
The equality $W(G)=\alpha(G)$ can be analyzed using a vertex $v$ of
minimum degree $\delta$. Its closed neighborhood $N[v]$ contains
$\delta+1$ vertices, and the remaining vertices induce $H$.
Each vertex of $N[v]$ contributes at most $1/(\delta+1)$
to $W(G)$. Deleting $N[v]$ cannot decrease the reciprocal-degree terms
on the surviving vertices. Applying Caro--Wei to $H$ and adjoining
$v$ to a maximum independent set of $H$ therefore gives
\[
  W(G)\le1+W(H)\le1+\alpha(H)\le\alpha(G).
\]

\medskip\noindent
\emph{Clique rigidity.}
Since $W(G)=\alpha(G)$, both degree comparisons in the preceding
argument must be tight. Every vertex of $N[v]$ therefore has degree
$\delta$, and no degree in $H$ can fall when $N[v]$ is deleted.
Consequently, no edge joins $N[v]$ to $H$. The $\delta+1$ vertices of
$N[v]$ all have degree $\delta$, so they form a clique. Connectedness then makes this clique the
whole graph, the classical connected equality case of the Caro--Wei
bound~\cite[Section~1.2]{Boppana2020}. Since $\Delta=3$, it is $K_4$.

\medskip\noindent
\emph{Degree-two reduction.}
At maximum degree two, Theorem~\ref{thm:main}'s proof gives
$G=P_n$ or $G=C_n$ for some $n\ge3$, with $a(G)=\alpha(G)$.
The equality criterion therefore reduces to $R(G)=\alpha(G)$.
The residue can be computed by tracking the numbers $k$ of twos and $h$
of ones in the current degree list.
For $k\ge3$, a Havel--Hakimi step removes one two and changes two
other twos to ones, creating no zeros. Paths start with two ones,
and the first step on a cycle also creates two ones. The subsequent
steps keep $h$ even and leave $h\ge2$ when the remaining $k$ is at
most two. Thus the terminal reduction at $k=1$ has two ones to
decrement. The transition and its preserved sum are
\begin{align*}
  (k,h)&\longmapsto(k-3,h+2),\\
  \frac{k-3}{3}+\frac{h+2}{2}&=\frac k3+\frac h2.
\end{align*}

\medskip\noindent
\emph{Terminal residue counts.}
In the terminal case $k=0$, reducing the ones in pairs leaves one zero per pair.
For $k=1$, removing the two leaves two zeros and $h-2$ ones.
For $k=2$, removing one two decrements the other two to a one and an
existing one to a zero, leaving one zero and $h$ ones.
In each case, pairing the remaining ones gives a terminal residue
equal to the preserved ceiling.
\[
  \left\lceil\frac k3+\frac h2\right\rceil=
  \begin{cases}
    \frac h2,&k=0,\\[2pt]
    \frac h2+1,&k=1,\\[2pt]
    \frac h2+1,&k=2.
  \end{cases}
\]

\medskip\noindent
Since the earlier steps create no zeros, the terminal count is the
residue of the original list. The preserved ceiling can therefore be
evaluated at the initial counts $(n-2,2)$ for paths and $(n,0)$ for
cycles, yielding the degree-two residue formulas
of~\cite[Theorem~1.5]{Favaron1991}.
\begin{align*}
  R(P_n)&=\left\lceil\frac{n-2}{3}+1\right\rceil
    =\left\lceil\frac{n+1}{3}\right\rceil,\\
  R(C_n)&=\left\lceil\frac n3\right\rceil.
\end{align*}

\medskip\noindent
\emph{Degree-two order exclusion.}
Comparing these residues with the independence numbers identifies the
possible equality orders.
The proof of Theorem~\ref{thm:main} gives
$\alpha(P_n)=\lceil n/2\rceil$ and $\alpha(C_n)=\lfloor n/2\rfloor$.
The ceiling of an integer divided by three exceeds that quotient by at
most two thirds. These rounding bounds exclude paths with $n\ge7$ and cycles with
$n\ge8$ from equality, by the strict comparisons
\begin{align*}
  R(P_n)&\le\frac{n+3}{3}<\frac n2\le\alpha(P_n),\\
  R(C_n)&\le\frac{n+2}{3}<\frac{n-1}{2}\le\alpha(C_n).
\end{align*}

\medskip\noindent
\emph{Remaining equality cases.}
For the remaining path orders $3,4,5,6$, substitution in the residue and
independence formulas gives equality at $3,4,6$; the path $P_5$ has
residue two and independence number three. For the remaining cycle
orders $3,4,5,6,7$, equality holds at $3,4,5,7$; the cycle $C_6$
also has residue two and independence number three. Including $K_4$
from the higher-degree case gives exactly the eight stated graphs.
Their maximum degrees are at most three, so the exhaustive classification
also proves strictness for $\Delta\ge4$.
\end{proof}

The classification determines exactly when the TxGraffiti inequality is
sharp. Sharpness alone does not establish an improved independence
estimate; that requires comparison with the residue.

\begin{corollary}[Domination by the residue]\label{cor:dominated}
Every finite simple graph $G$ with maximum degree $\Delta\ge3$
satisfies $(a(G)+R(G))/\Delta\le R(G)$.
Equality is attained by $K_4$.
\end{corollary}

\begin{proof}
Theorem~\ref{thm:vehicle} bounds $a(G)$ by $(\Delta+1)W(G)/2$.
Adding $R(G)$ to this bound and dividing by the positive maximum degree
$\Delta$ gives a bound on the TxGraffiti estimate.
The classical inequality $W(G)\le R(G)$ then expresses that bound solely
in terms of the residue. Since $(\Delta+3)/(2\Delta)\le1$ for
$\Delta\ge3$, the resulting coefficient completes the comparison.
\begin{align*}
  \frac{a(G)+R(G)}{\Delta}
    &\le\frac{\Delta+1}{2\Delta}W(G)+\frac{R(G)}{\Delta}\\
    &\le\frac{\Delta+3}{2\Delta}R(G)\\
    &\le R(G).
\end{align*}

\medskip\noindent
Equality follows from the values for $K_4$ used in the proof of
Theorem~\ref{thm:main}, namely $a(K_4)=2$, $R(K_4)=1$ and $\Delta(K_4)=3$.
The quotient is therefore one, giving equality.
\end{proof}

Alongside the equality examples, a finite computational check tested
Theorem~\ref{thm:main}.
It found no violations among all $273{,}191$
connected graphs of orders $3$ through $9$.\footnote{\url{https://github.com/ChakshuGupta13/lab/tree/main/math/gupta2026annihilation}.}
The graphs were generated with \texttt{geng} from
\texttt{nauty}~\cite{McKay2013nauty}. For each order, the generator's
successful exit was checked separately and its graph count was compared
with OEIS A001349.\footnote{\url{https://oeis.org/A001349}.}
The independence number was
computed exactly by branching on inclusion or exclusion of a vertex,
with available vertices represented by bit sets. The annihilation
number and residue were computed from the degree sequence. Comparing
the resulting integers supplied a computational check of the proved
inequality, not an alternative proof.

\section{Conclusion}\label{sec:disc}

For graphs with an edge, the residue and annihilation number give an
independence interval contained in the Caro--Wei interval. Strict
improvement of both endpoints is compatible with the same sharp worst-case
ratio, so containment is a gain that the ratio alone does not capture.
The TxGraffiti application makes the analogous distinction between proving
a lower estimate and improving an existing one. For maximum degree at
least three, its estimate is valid but never stronger than the residue,
even without connectedness.

\section*{Declarations}

No funding was received for this work. The author declares no conflicts of interest.

\section*{Use of Artificial Intelligence}

GitHub Copilot assisted with developing and checking proof arguments,
writing and running computational verification code, locating sources,
and drafting and revising the manuscript. Automated correctness and
readability reviews were also used to identify possible defects.
The author directed the work and retains responsibility for the
mathematical claims, citations and integrity of the manuscript.

\begingroup
\raggedright
\bibliographystyle{amsplain}
\bibliography{references}

\providecommand{\bysame}{\leavevmode\hbox to3em{\hrulefill}\thinspace}
\providecommand{\MR}{\relax\ifhmode\unskip\space\fi MR }
% \MRhref is called by the amsart/book/proc definition of \MR.
\providecommand{\MRhref}[2]{%
  \href{http://www.ams.org/mathscinet-getitem?mr=#1}{#2}
}
\providecommand{\href}[2]{#2}
\begin{thebibliography}{1}

\bibitem{Boppana2020}
Ravi~B. Boppana, Magn{\'u}s~M. Halld{\'o}rsson, and Dror Rawitz, \emph{Simple
  and local independent set approximation}, Theoret. Comput. Sci. \textbf{846}
  (2020), 27--37, \url{https://doi.org/10.1016/j.tcs.2020.09.018}.

\bibitem{Davila2025reverie}
Randy Davila, Boris Brimkov, and Ryan Pepper, \emph{In reverie together: Ten
  years of mathematical discovery with a machine collaborator}, 2025, Preprint,
  arXiv:2507.17780v1, \url{https://arxiv.org/abs/2507.17780v1}.

\bibitem{Favaron1991}
Odile Favaron, Maryvonne Mah{\'e}o, and Jean-Fran{\c{c}}ois Sacl{\'e}, \emph{On
  the residue of a graph}, J. Graph Theory \textbf{15} (1991), 39--64,
  \url{https://doi.org/10.1002/jgt.3190150107}.

\bibitem{Hakimi1962}
S.~L. Hakimi, \emph{On realizability of a set of integers as degrees of the
  vertices of a linear graph. {I}}, J. Soc. Indust. Appl. Math. \textbf{10}
  (1962), no.~3, 496--506, \url{https://doi.org/10.1137/0110037}.

\bibitem{Havel1955}
V{\'a}clav Havel, \emph{Pozn{\'a}mka o existenci kone{\v c}n{\'y}ch graf{\r
  u}}, {\v C}asopis pro p{\v e}stov{\'a}n{\'\i} matematiky \textbf{80} (1955),
  no.~4, 477--480 (Czech), \url{https://dml.cz/dmlcz/108220}.

\bibitem{KadrawiLevit2026}
Ohr Kadrawi and Vadim~E. Levit, \emph{Annihilation, independence, and residue:
  Sharp matching bounds for the annihilation gap and a {TxGraffiti}
  application}, 2026, Preprint, arXiv:2607.01438v1,
  \url{https://arxiv.org/abs/2607.01438v1}.

\bibitem{McKay2013nauty}
Brendan~D. McKay and Adolfo Piperno, \emph{Practical graph isomorphism, {II}},
  2013, Preprint, arXiv:1301.1493v1, \url{https://arxiv.org/abs/1301.1493v1}.

\bibitem{Pepper2004}
Ryan~D. Pepper, \emph{Binding independence}, {Ph.D.} dissertation, University
  of Houston, 2004, UMI No.~3156029,
  \url{https://www.proquest.com/docview/305196562}.

\end{thebibliography}
\endgroup

\end{document}